\documentclass[11pt]{article}

\usepackage[margin=1.5in]{geometry}

\usepackage{authblk, amsfonts, amsmath, amssymb, amsthm, bbm, xcolor, mathrsfs, mathtools}
\usepackage{graphicx} %For inserting images
\usepackage{graphicx, tikz}%For figures
\usetikzlibrary{decorations.markings}
\usetikzlibrary{arrows.meta}
\usepackage{enumitem} %For changing item separation
\usepackage{hyperref} %so that you can click on theorem numbers
\usepackage[textwidth=1.5in]{todonotes} %For comment formatting
\usepackage{thm-restate}
\usepackage{cleveref}
\usepackage{comment}

\newtheorem{thm}{Theorem}[section]

\newtheorem{lemma}[thm]{Lemma}
\newtheorem{observation}[thm]{Observation}

\theoremstyle{definition}
 
\newenvironment{claimth}[1]{\par\noindent\underline{Claim:}\space#1}{}
\newenvironment{claimproof}[1]{\par\noindent\underline{Proof:}\space#1}{\hfill $\blacksquare$}

\newcommand{\C}{\mathcal{C}}
\DeclareMathOperator{\val}{val}
\newcommand{\rank}{\mathrm{rank}}

\definecolor{DeepCarrotOrange}{rgb}{0.91, 0.41, 0.17}
\definecolor{BananaYellow}{rgb}{1.0, 0.88, 0.21}

\definecolor{CornflowerBlue}{rgb}{0.39, 0.58, 0.93}
\definecolor{Magenta}{rgb}{0.50, 0.0, 0.50}
\definecolor{AppleGreen}{rgb}{0.55, 0.71, 0.0}
\definecolor{AO}{rgb}{0.0, 0.5, 0.0}

\setlist[enumerate]{itemsep=0em}
\tikzset{MyNode/.style={circle, draw, inner sep=2,outer sep=0, fill=gray}}
\tikzset{MyRedArc/.style={line width=2.3pt, red}}
\tikzset{MyBlueArc/.style={line width=2.3pt, blue}}
\tikzset{
  arrHalfway/.style={
    postaction={
      very thick,
      decorate,
      decoration={
        markings,
        mark= at position .56 with {\arrow{>}},
}}}}
\tikzset{
  arrHalfwayLoop/.style={
    postaction={
      very thick,
      decorate,
      decoration={
        markings,
        mark= at position .5 with {\arrow{<}},
}}}}

\title{Full immersions in Eulerian group-labeled graphs}
\author{Ioannis Eleftheriadis}
\affil{Magdalen College, University of Oxford}
\author{Rose McCarty\thanks{Supported by the National Science Foundation under Grant No.~DMS-2452111.}}
\affil{Schools of Mathematics and Computer Science, Georgia Institute of Technology, USA.}
\author{Caleb McFarland\thanks{Supported in part by the Georgia Tech ARC-ACO Fellowship and in part by the National Science Foundation under Grant No.~DMS-2452111.}}
\affil{School of Mathematics, Georgia Institute of Technology, USA.}
\author{Evangelos Protopapas\thanks{Supported by the ERC project BUKA (n° 101126229).}}
\affil{Faculty of Mathematics, Informatics and Mechanics, University of Warsaw, Poland.}
\author{Paul Wollan}
\affil{Sapienza Universit\`{a} di Roma, Italy.}

\date{October 2026}

\begin{document}

\maketitle

\begin{abstract}
An immersion of a graph $H$ in a graph $G$ is \textit{full} if the immersion uses every edge of $G$.  We give a structure theorem for group-labeled Eulerian graphs excluding a fixed group-labeled graph as a full immersion.  
\end{abstract}

\section{Introduction}
In this paper, we extend the well-developed theory of structure theorems for graphs excluding a fixed immersion to what we call \textit{full immersions} in group labeled graphs.  A primary goal in doing so is to develop tools for analyzing vertex minors in graphs which are obtained from circle graphs by adding a fixed number of vertices arbitrarily attached to the rest of the circle graph.  

All graphs in this paper are finite and possibly have both loops and parallel edges.  A graph $G$ contains a graph $H$ as an \textit{immersion} if there exists a function $\phi$ whose domain is $V(H) \cup E(H)$ such that:
\begin{enumerate}
    \item $\phi(v) \in V(G)$ and $\phi(v) \neq \phi(u)$ for all $u,v \in V(H), u \neq v$;
    \item for all edges $e \in E(H)$ with distinct ends $u$ and $v$, $\phi(e)$ is a walk with endpoints $\phi(u)$ and $\phi(v)$; for all loops $e\in E(H)$ attaching to the vertex $v$, $\phi(e)$ is a non-empty closed walk which starts and ends at $\phi(v)$; 
    \item For every edge $e \in E(H)$, the walk $\phi(e)$ uses each edge at most once; for all distinct edges $e, f \in E(H)$, the walks $\phi(e)$ and $\phi(f)$ are pairwise edge disjoint.
\end{enumerate}
The vertices $\{\phi(v): v \in V(H)\}$ are the \textit{branch vertices} of the immersion, and the walks $\phi(e)$ for $e \in E(H)$ are denoted the \textit{branch walks}.  In the literature, a distinction is made between \textit{weak} and \textit{strong} immersions, where strong immersions have the additional requirement that no branch vertex be contained as an internal vertex of a branch walk.  However, for our purposes here, we make no such distinction and use immersion to refer to what has historically been called weak immersion.  Immersions were introduced by Nash-Williams in \cite{NashWilliams1964} where he conjectured that graphs are well quasi-ordered under immersion containment.  As part of their graph minors series, Robertson and Seymour affirmed the conjecture \cite{ROBERTSON2010181} and general structure theorems for graphs excluding a fixed immersion were independently proven in~\cite{Devos2013note,WOLLAN2015}.  

We present a structure theorem for excluded \textit{full} immersions in \textit{group labeled graphs}.  An immersion of a graph $H$ in a graph $G$ is \textit{full} if every edge of $G$ is contained in the union of the edge sets of the branch walks.    

Given a group $\Gamma$, a \textit{$\Gamma$-labeled graph} is an undirected graph where each edge $e$ is assigned both an orientation and a label $\alpha_e$ in $\Gamma$.  Informally, when considering immersions in group-labeled graphs, the sequence of edges in a branch walk has a weight in $\Gamma$ obtained by the product of the weights on the edges of the walk, where for each edge e on the walk, we right-multiply by $\alpha_e$ when traversing an edge $e$ according to the orientation, and right-multiply by $\alpha_e^{-1}$ when traversing the edge contrary to the orientation. Some care is needed to rigorously deal with the orientations and we give the exact definition below.    
Group-labeled graphs have been used in a number of contexts: graph embeddings and coverings \cite{GROSS1974239, GROSS1977273}, symmetry-forced rigidity \cite{Bernstein2022190, Kaszanitzky2021176, Tanigawa15} (i.e., bar and joint frameworks where every non-rigid motion breaks a specified symmetry), as well as structural \cite{GGW14, GerardsGraphicMatroids, SeymourGraphicMatroids} and extremal \cite{NW22, PSWX24} properties of matroids. 

A primary motivation for studying full immersions in Eulerian group-labeled graphs comes from the structure theory of excluded vertex minors. A \textit{circle graph} is the intersection graph of a family of chords of a circle. Circle graphs play a central role in the area, including Bouchet's characterization by excluded vertex minors \cite{Bouchet1994} and the vertex-minor grid theorem \cite{VertexGrid}.  Equivalently, for a circle graph $G$ there exists a 4-regular multigraph $\bar{G}$ with $V(G) = V(\bar{G})$ for which a fixed Eulerian tour encodes the adjacencies in $G$: two vertices are adjacent in $G$ exactly when their occurrences alternate along the tour \cite{Bouchet1972,ReadRosenstiehl1978}. In this representation, local complementation at a vertex $v$ corresponds to changing the Euler tour by rerouting at the vertex $v$, and vertex-minor containment in $G$ is reflected by full immersion in the corresponding $4$-regular graphs. This correspondence extends naturally to graphs obtained from circle graphs by adding vertices with arbitrary neighborhoods. Consider a graph $G^+$ obtained by adding a vertex $x$ to $G$ with an arbitrary neighborhood in $G$.  The neighborhood of $x$ can be encoded by a $\mathbb{Z}_2$-labeling of the edges of $\bar{G}$, with adjacency determined by the parity of the sum of the labels along the portion of the Euler tour between the two occurrences of a vertex. Labeled full immersions in the labeled $\bar{G}$ then allow us to study vertex-minors within $G^+$. For graphs obtained by adding $k$ vertices to the circle graph $G$, we can again encode their neighborhoods coordinatewise by labels in $\mathbb{Z}_2^k$ in $\bar{G}$. Thus, full immersions in group-labeled graphs provide a framework for studying vertex minors in graphs which are nearly circle graphs.

%PRELIMINARIES
%\section{Preliminaries}
%\label{sec:prelim}

In general, we will follow standard graph theoretic notation.  
Given a graph $G$ and a set of vertices $B$, we write $\delta_G(B)$ (or simply $\delta(B)$ when $G$ is clear from the context) for the set of all edges of $G$ with exactly one end in $B$.  
If $B$ consists of a single vertex $v$, we write $\delta(v)$.  We will use $E(B)$ to refer to the set $E(G[B])$.  For a set $X$ and $x \in X$, we will use the notation $X - x$ for $X \setminus \{x\}$.  

\subsection{Group-labeled Graphs}

We now introduce the notation to work with orientations of edges in group labeled graphs.
Given a graph $G$, for each edge $e$ of $G$ we define two distinct \textit{oriented edges} which we denote by $\vec{e}$ and $\vec{e}^{-1}$.  
We call $e$ the \textit{underlying} edge of an oriented edge $\vec{e}$ and will refer to $\vec{e}$ as an \textit{orientation} of the edge $e$.
An oriented edge $\vec{e}$ has a \emph{tail} and a \emph{head}, each of which is an endpoint of the underlying edge $e$.  If $e$ has two distinct endpoints, then the head and tail are distinct vertices.  Moreover, it holds for all oriented edges $\vec{e}$ that the head of $\vec{e}$ is equal to the tail of $\vec{e}^{-1}$.  We say that $\vec{e}$ is oriented \emph{from} its tail \emph{to} its head.  We write $\vec{E}(G)$ for the set of all oriented edges arising from the edges of $G$.  Depending on the context, we will want to consider the degree of a vertex either taking into account possible loops or not.  To that end, we adopt the following convention.  The \textit{degree} of a vertex $v$, denoted $\deg(v)$, is $|\delta(v)|$ plus twice the number of loops incident to $v$.  We will use $d(v)$ as notation to indicate $|\delta(v)|$.  Note that $\deg(v)$ and $d(v)$ always have the same parity, regardless of the presence of loops.

%RMCOMMENT{I was less formal about orientations because we need loops to have two different orientations in order to define the labeling.}

%PWCOMMENT{I've expanded the formalism slightly so that it still requires loops to have two distinct orientations.}

%RMCOMMENT{Changed below so that $\gamma'$ is a shifting of $\gamma$ rather than $(G, \gamma')$ is a shifting of $(G, \gamma)$.}

Let $\Gamma$ be a group.
We denote the group operation multiplicatively and the identity by 1.
A \emph{$\Gamma$-labeling} of a graph $G$ is a function $\gamma\colon \vec{E}(G) \rightarrow \Gamma$ such that $\gamma(\vec{e}^{-1}) = \gamma(\vec{e})^{-1}$ for any edge $\vec{e}$ of $G$.
We call $\gamma(\vec{e})$ the \textit{label} of an oriented edge $\vec{e}$.  Throughout the article, we adopt the convention that $\gamma(\emptyset) = 1$ for any labeling function $\gamma$.  
A \textit{$\Gamma$-labeled graph} is a pair $(G, \gamma)$ where $G$ is a graph and $\gamma$ is a $\Gamma$-labeling of $G$.
There is an equivalence relation on $\Gamma$-labelings defined as follows.
Given a $\Gamma$-labeled graph $(G,\gamma)$, a vertex $x \in V(G)$, and a group element $\alpha \in \Gamma$, \textit{shifting by $\alpha$ at $x$} means to append $\alpha$ at the beginning of the label of each oriented edge with tail $x$ and to append $\alpha^{-1}$ at the end of the label of each oriented edge with head $x$.
Thus, if $\vec{e}$ is an oriented loop at $x$, then its new label after shifting is $\alpha\gamma(\vec{e})\alpha^{-1}$.
Likewise, if $\vec{e}$ is an oriented edge whose tail is $x$ and whose head is not $x$, then its new label is $\alpha\gamma(\vec{e})$.
This is the inverse of $\gamma(\vec{e}^{-1})\alpha^{-1}$, which is the new label of $\vec{e}^{-1}$.
We say that a $\Gamma$-labeling $\gamma'$ of $G$ is a \emph{shifting} of $\gamma$ if it can be obtained by a (finite) sequence of such shiftings.

%RMCOMMENT{I put $+\alpha$ in front and $-\alpha$ in back so that it is conjugation by $\alpha$ on loops.}

A \textit{trail} $T$ in a graph $G$ is a sequence of oriented edges $\vec{e_1}\vec{e_2}\dots\vec{e_n}$, $n\ge1$, such that the underlying edges are all distinct, and for each $i \in [n-1]$, the head of $\vec{e_i}$ equals the tail of $\vec{e}_{i+1}$. Note that the endpoints of the edges need not be distinct; the same vertex may be visited multiple times. The  \textit{tail} of $T$ is the tail of $\vec{e}_1$ and the \textit{head} of $T$ is the head of $\vec{e_n}$. We call the head and tail of $T$ its \textit{ends}. We use $V(T)$ to refer to the set of vertices which are either heads or tails of oriented edges in $T$.  We write $E(T)$ to refer to the undirected edges of $G$ which are the underlying edges of oriented edges contained in $T$.  We say two trails $T_1$ and $T_2$ are \textit{edge disjoint} if $E(T_1) \cap E(T_2) = \emptyset$.  Given a group-labeling $\gamma$ of $G$, the \emph{label} of $T$ is $\gamma(T) = \gamma(\vec{e_1})\gamma(\vec{e_2})\ldots\gamma(\vec{e_n})$. The \emph{inverse} of $T$ is the trail $T^{-1}$ which is the sequence of edges in the reverse order and with the reverse orientation. Thus $\gamma(T^{-1})=\gamma(T)^{-1}$. The \textit{transitions} of $T$ are the ordered pairs $(\vec{e_i}, \vec{e_{i+1}})$ for $i \in [n-1]$. Thus $T$ has $n-1$ transitions. We view transitions of the form $(\vec{e_1}, \vec{e_{2}})$ and $(\vec{e_2}^{-1}, \vec{e_{1}}^{-1})$ as being equivalent; thus $T^{-1}$ has the same transitions as $T$.

%RMCOMMENT{I changed the definition of transitions because I found this clearer, but also because it was strange that if $e_1$ is a loop then $\vec{e_1}\vec{e_2}$ and $\vec{e_1}^{-1}\vec{e_2}$ have the same transition.}

%RMCOMMENT{I guess shifting needs to be added to the definitions below; we really don't want ``admitting an immersion'' to not be invariant under shifting.}

\subsection{Immersions}
\label{subsec:Immersions}

We now extend the definition of immersions to group-labeled graphs.  

Let $\Gamma$ be a group. We say that a $\Gamma$-labeled graph $(G, \gamma_G)$ admits an \textit{immersion} of a $\Gamma$-labeled graph $(H, \gamma_H)$ if there exist a shifting $\gamma_G'$ of $\gamma_G$ and a function $\varphi$ which injectively maps $V(H)$ to $V(G)$ and which maps $\vec{E}(H)$ to trails in $G$ such that\begin{enumerate}
    \item for any oriented edge $\vec e \in \vec{E}(H)$, the trails $\varphi(\vec{e})$ and $\varphi(\vec{e}^{-1})$ are inverses,
    \item for any distinct edges $e_1$ and $e_2$ of $H$, the trails $\varphi(\vec{e_1})$ and $\varphi(\vec{e_2})$ are edge-disjoint, and
    \item for any oriented edge $\vec{e}$ of $H$, the tail of $\vec{e}$ is mapped to the tail of $\varphi(\vec{e})$, the head of $\vec{e}$ is mapped to the head of $\varphi(\vec{e})$, and $\gamma_H(\vec{e}) = \gamma_G'(\varphi(\vec{e}))$.
\end{enumerate} We say that $\varphi$ is an \emph{immersion} of $(H, \gamma_H)$ \emph{into} $(G, \gamma_G)$. As in the case of undirected graphs, we call the image of $V(H)$ under $\varphi$ the \textit{branch vertices} and the image of $\vec{E}(H)$ under $\varphi$ the \textit{branch trails}. The \textit{transitions} of an immersion are the transitions of its branch trails. 
The labeling $\gamma_G'$ is an \textit{accommodating labeling} of the immersion; in general, an accommodating labeling will not necessarily be unique.  

A graph $G$ contains a graph $H$ as a \emph{near-full immersion} if there is an $H$-immersion $\phi$ such that there is at most one non-trivial component $C$ after deleting edges of $\bigcup_{e \in E(H)} E(\phi(e))$ in $G$, and moreover, the component $C$, if it exists, contains a branch vertex of $H$. For a vertex $x \in V(H)$, if the component $C$, when it exists, contains the vertex $\phi(x)$, we say $G$ contains $H$ as an \emph{$x$-near-full immersion}.  The definition extends analogously to group labeled immersions.

\begin{observation}\label{obs:transitivity of almost total immersions}
    If $(G_1, \gamma_1)$ has an $x$-near-full immersion $\varphi_2$ of $(G_2, \gamma_2)$, and $(G_2, \gamma_2)$ has a $y$-near-full immersion $\varphi_3$ of $(G_3, \gamma_3)$ with $\varphi_3(y) = x$, then $(G_1, \gamma_1)$ has a $y$-near-full immersion of $(G_3, \gamma_3)$.
\end{observation}

Our primary tool in proving the structure theorem for excluded full immersions in Eulerian group labeled graphs is the following theorem which says that excluding a fixed near-full immersion implies excluding some larger group labeled immersion with no further requirements on the edge set which it utilizes.  

\begin{restatable}{thm}{mainReduction}
\label{thm:mainReduction}
    Let $\Gamma$ be a finite group. For every $\Gamma$-labeled graph $(H, \gamma_H)$ there exists a $\Gamma$-labeled graph $(H', \gamma_{H'})$ such that for every Eulerian $4$-edge-connected $\Gamma$-labeled graph $(G, \gamma_G)$ and for every vertex $x \in V(H)$, if $(G, \gamma_G)$ contains $(H', \gamma_{H'})$ as an immersion, then $(G, \gamma_G)$ contains $(H,\gamma_H)$ as an $x$-near-full immersion.
\end{restatable}

\subsection{The structure theorem}
\label{subsec:structureThm}

The structure theorem for Eulerian group-labeled graphs excluding a fixed graph as a full immersion builds on the structure of graphs excluding a group-labeled immersion.  We give the notation necessary to state the structure theorem for excluded group-labeled immersions (without the requirement that they be nearly full), and then present the statement for excluded full group-labeled immersions.

A \textit{tree-cut decomposition} of a graph $G$ is a pair $(T, \mathscr{B})$ such that $T$ is a tree and $\mathscr{B} = (B_t : t \in V(T))$ is a collection of pairwise-disjoint (possibly empty) subsets of $V(G)$ whose union is $V(G)$.  The sets in $\mathscr{B}$ are the \textit{bags} of the decomposition. 

Roughly, the structure theorem for excluded group-labeled immersions says that each bag either 1) has a ``torso'' with few high degree vertices, or 2) has ``almost all of its edges'' labeled over a proper subgroup. 

So, let $G$ be a graph, and let $(T, \mathscr{B})$ be a tree-cut decomposition of $G$. Given a bag $B_t \in \mathscr{B}$, the \textit{torso} of $B_t$ is the graph $H$ obtained from $G$ as follows. Let $T_1, \dots, T_\ell$ be the components of $T \setminus t$. For each $i \in [\ell]$, let $V_i = \bigcup_{s \in V(T_i)} B_{s}$. Then $H$ is created by identifying each non-empty set $V_i$ to a single vertex, which we call a \emph{new vertex} of the torso. We keep parallel edges, but we remove any loops at the new vertices. 

Now, let $\Gamma$ be a group, let $(G, \gamma)$ be a $\Gamma$-labeled graph, and let $B\subseteq V(G)$. Then a \textit{certificate for $B$} is a tuple $(X,\gamma')$ so that $X$ is a set of edges of $G$ with $\delta(B) \subseteq X \subseteq \delta(B) \cup E(B)$ and $\gamma'$ is a shifting of $\gamma$ so that the $\gamma'$-labels of orientations of the edges in $E(B)\setminus X$ generate a proper subgroup of~$\Gamma$. The \emph{value} of a certificate is
\begin{align*}
    \val(X,\gamma') &= |\delta(B)| + 2|X \cap E(B)|.
\end{align*}This counts the edges of $X$, except edges with both ends in $B$ are counted twice. 

\begin{restatable}[\cite{McCartyMW2025Structure}]{thm}{structtheorem1}
\label{thm:MainStructureThm}
    For any finite group $\Gamma$ and positive integers $k$ and $n$, there is an integer $t=4kn|\Gamma|^{6 + \lfloor\log_2|\Gamma|\rfloor}$ so that the following holds for any 2-edge-connected $\Gamma$-labeled graph $(G, \gamma)$ which forbids an immersion of some $n$-vertex $\Gamma$-labeled graph where each vertex is incident to at most $k$ edges. 
    
    There exist a shifting $\gamma'$ of $\gamma$ and a tree-cut decomposition $(T, \mathscr{B})$ of $G$ such that for every bag $B \in \mathscr{B}$, at least one of the following holds.
    \begin{enumerate}
        \item There are at most $n|\Gamma|$ vertices in the torso of $B$ of degree greater than $t$, and none of them are the new vertices of the torso.
        \item There exists $X \subseteq E(G)$ so that $(X,\gamma')$ is a certificate for $B$ of value at most~$t$.
    \end{enumerate}
\end{restatable}

We are now able to state the structure theorem for excluded \textit{full} group labeled immersions in Eulerian graphs.  The theorem finds the same structure as
    Theorem \ref{thm:MainStructureThm} unless the total edge labels of $G$ and $H$ differ
    in an elementary abelian $2$-group quotient of $\Gamma$.

\begin{restatable}{thm}{fullStructure}
\label{thm:fullStructure}
    Let $\Gamma$ be a finite group. For every Eulerian
    $\Gamma$-labeled graph $(H,\gamma_H)$ with at least one edge,
    there exists a positive integer $t = t(\Gamma, H)$ such that the following holds.
    Let $(G,\gamma_G)$ be a $4$-edge-connected Eulerian
    $\Gamma$-labeled graph with no $(H,\gamma_H)$ full immersion.
    Then one of the following holds.
    \begin{enumerate}
        \item There exists a normal subgroup $\Gamma'$ of $\Gamma$
        such that $\Gamma/\Gamma'$ is isomorphic to $\mathbb{Z}_2^r$
        for some nonnegative integer $r$, and over this quotient
        labeling the sum of the edge labels in $(G,\gamma_G)$
        differs from that in $(H,\gamma_H)$.

        \item There exist a shifting $\tilde{\gamma}$ of $\gamma_G$
        and a tree-cut decomposition $(T,\mathcal{B})$ of $G$
        such that for every bag $B \in \mathcal{B}$, at least one
        of the following holds.
        \begin{enumerate}
            \item There are at most $t$ vertices in the torso of $B$
            of degree greater than $t$, and none of them are
            the new vertices of the torso.

            \item There exists $X \subseteq E(G)$ such that
            $(X,\tilde{\gamma})$ is a certificate for $B$
            of value at most $t$.
        \end{enumerate}
    \end{enumerate}
\end{restatable}

The proof of Theorem \ref{thm:fullStructure} is presented in Section \ref{sec:ExtendingNear}.

\subsection{Outline}

The paper is organized as follows.  In Section \ref{sec:packing}, we present Theorem \ref{thm:EP} showing that if a 4-edge-connected $\Gamma$-labeled Eulerian graph has many edge disjoint circuits sharing a common vertex $x$ each with label not in a fixed subgroup $\Gamma'$ of $\Gamma$, then the edge set of the whole graph can be partitioned into edge disjoint circuits through $x$ each with label outside $\Gamma'$.  This packing result is used in Section \ref{sec:NearFull} to prove Theorem \ref{thm:mainReduction} which states that excluding a fixed group-labeled $(H, \gamma_H)$ as a near-full immersion in a 4-edge-connected Eulerian graph $(G, \gamma_G)$ implies that $(G, \gamma_G)$ does not contain a larger $(H', \gamma_{H'})$ as a group labeled immersion without the requirement that it be nearly full.  Already, Theorem \ref{thm:mainReduction} together with Theorem \ref{thm:MainStructureThm} yields a structure theorem for excluded near-full immersions in 4-edge-connected group-labeled Eulerian graphs.  Finally, in Section \ref{sec:ExtendingNear}, we determine when a near-full immersion can be
extended to a full immersion, providing the final ingredient needed
to prove Theorem~\ref{thm:fullStructure}.

%%%%%%%%%%%%%%%%%%%
%PACKING NONZERO CIRCUITS
%%%%%%%%%%%%%%%%%%%%

\section{Decomposing to circuits labeled outside of a subgroup}
\label{sec:packing}

Let $\Gamma$ be a group, let $\Gamma'$ be a subgroup of $\Gamma$, let $(G, \gamma)$ be a $\Gamma$-labeled graph, and let $x$ be a vertex of $G$.

A \emph{circuit} in $G$ is a trail whose head and tail are equal; we say that it \emph{begins} and \emph{ends} at that vertex.
Given a vertex $x$ of $G$, an \emph{$x$-packing} of $G$ is a collection $\mathcal{C}$ of pairwise edge-disjoint circuits so that each circuit begins and ends at $x$.  We say that the $x$-packing \textit{saturates $x$} if $|\C| = \deg (x)/2$.  
We call an $x$-packing an \emph{$(x, \Gamma')$-packing} if every circuit has label outside of $\Gamma'$.
Similarly, an \emph{$(x, \Gamma')$-decomposition} is an $(x, \Gamma')$-packing $\mathcal{C}$ so that every edge of $G$ is in $E(T)$ for some $T \in \mathcal{C}$.

This section is dedicated to proving the following theorem.

\begin{restatable}{thm}{decomptocircuits}
\label{thm:EP}
Let $k \ge 1$ be a positive integer.  Let $\Gamma$ be a group, let $\Gamma'$ be a subgroup of $\Gamma$, let $(G, \gamma)$ be a $\Gamma$-labeled Eulerian graph which is $4$-edge-connected, and let $x$ be a vertex of $G$ such that there exists an $(x, \Gamma')$-packing of size $k$.
Then there exists an $(x, \Gamma')$-decomposition of size at least $\lceil k/2\rceil$.
\end{restatable}

We proceed to introduce the set-up needed to establish \Cref{thm:EP} above.  
Fix an Eulerian, $4$-edge-connected $\Gamma$-labeled graph $(G,\gamma)$, a vertex $x \in V(G)$, and an $x$-packing $\C$ which saturates $x$.
We refer to each nontrivial connected component of $G - \bigcup_{C \in \C} E(C)$ as an \textit{unused component} of $\C$.  Note each unused component is an Eulerian subgraph of $G$.  
We write $\mathsf{comp}_G(\C)$ for the auxiliary bipartite graph whose vertices in one part are the circuits in $\C$, in the other the unused components of $\C$, and vertices in distinct parts are adjacent if and only if the corresponding parts share a vertex.
Whenever $G$ is clear from the context we drop the subscript.
Given a vertex $X$ of $\mathsf{comp}(\C)$ we write $\deg_\C(X)$ for the degree of $X$ in $\mathsf{comp}(\C)$; whenever we speak of the degree of a circuit or an unused component we implicitly mean their degree in $\mathsf{comp}(\C)$.

Let $H$ be an unused component of $\C$ of degree one, and $C \in \C$ the unique circuit which contains a vertex of $H$.
We define the \emph{closure of $H$}, denoted by $\mathrm{cl}(H)$, to be a containment minimal set $S \subseteq V(G) - x$ so that $V(H) \subseteq S$ and $\delta(S)$ contains at most two edges in $C$.
Observe that this notion is well-defined given $\C$ saturates $x$ and thus $V(G) -x$ contains $V(H)$ and $\delta(V(G) - x)$ intersects $C$ in exactly two edges.  
%\PW{I've assumed that the closure is only defined for packings which saturate x.  This solves two issues: one, in the case where the packing has size 1, an unused component H containing x can't be separated from x by an edge cut, and two, when C returns multiple times to x, it may not be the case that we can find an edge cut containing H which intersects C in only two edges.}
%\EP{This is not always the case. It is sufficient to say that $\C$ contains $\deg(x)/2$ many circuits. Maybe we should assume this here or figure out another way to deal with this.} %Since we use this definition frequently, we omit the phrase ``of $H$ into $C$'' whenever $H$ and $C$ are clear from context. We first establish that this notion is well-defined. 
Moreover, if $S,S' \subseteq V(G)- x$ both contain $V(H)$ and satisfy $|\delta(S) \cap E(C) | = |\delta(S') \cap E(C)| = 2$ , then so does $S \cap S'$.
To see this, notice that $S \cap S'$ must contain $V(H)$ and not $x$, and if $E(C)$ has an edge in $\delta(S \cap S')$, then it has at least two edges in $\delta(S \cup S')$ because it is an $x$-circuit.  By submodularity of the cut-function, the number of edges of $C$ in $\delta(S \cap S')$ is at most two, as required.  It follows that $\mathrm{cl}(H)$ is uniquely defined.  Finally, we note that $\mathrm{cl}(H) \subseteq V(H) \cup V(C)$ as deleting any vertex not contained in either $H$ or $C$ does not change how many boundary edges intersect $E(C)$.  

We call an unused component $H$ of $\C$ of degree one \emph{flexible} if the unique circuit $C \in \C$ incident to $H$ in $\mathsf{comp}(\C)$ is also of degree one in the auxiliary graph.
Symmetrically, we call $H$ \emph{inflexible} if $C$ has degree at least two.
An important consequence of the former definition is that $\mathrm{cl}(H)$ always intersects some other circuit in the packing whenever $H$ is flexible.
We formalize this in the lemma below.

\begin{lemma}
\label{lem:rankwelldfn}
Let $(G, \gamma)$ be a $4$-edge-connected Eulerian group-labeled graph, $x \in V(G)$ a vertex of $G$, $\C$ an $x$-packing which saturates $x$, and $H$ a flexible unused component of $\C$ that is adjacent to $C \in \C$ in $\mathsf{comp}(\C)$.
Then there is some $C'\in \C$ different from $C$ such that $\mathrm{cl}(H)\cap V(C')\neq \emptyset$.
\end{lemma}
\begin{proof}
Consider the cut $\delta(\mathrm{cl}(H))$. Since this contains at most two edges in $C$ and $G$ is $4$-edge-connected, there exists a trail $T$ in $G$ from $x$ to $\mathrm{cl}(H)$, internally disjoint from $\{x\}\cup\mathrm{cl}(H)$, that is edge-disjoint from $E(C) \cap \delta(\mathrm{cl}(H))$. Consider the last oriented edge $\vec e$ of $T$ with underlying edge $e$ and head $v$ contained in $\mathrm{cl}(H)$. The edge $e$ cannot lie in $E(C)$ by definition, nor in $H$ because $V(H) \subseteq \mathrm{cl}(H)$.  Evidently, the edge $e$ cannot lie in a different unused component of $\C$, since in that case $v$ must either lie in $V(C)$ implying that the degree of $C$ in $\mathsf{comp}(\C)$ is at least two, contradicting our assumptions, or lie in $V(H)$ contrary to $H$ being a component.
%\CM{currently it could be in a different unused component, because we don't know that one end of $\vec{e}$ is in $H$. You first have to assume that $C$ has degree 1, in which case it follows.}\EP{Indeed. Hopefully it's fixed now.}\PW{Agree that it is fine now.  Removing comments}
We conclude that there exists some other circuit $C' \in \C$ such that $e \in E(C')$ and thus $V(C')$ has a vertex in $\mathrm{cl}(H)$. 
\end{proof}

Dually to the closure, we consider the notion of an \emph{ear decomposition of $H$ into $C$}. This is a sequence $T_0:=H, T_1, \ldots, T_k$ so that
\begin{enumerate}
\item for each $i \in \{1,\ldots, k\}$, $T_i$ is a non-empty subtrail of $C-x$ that is edge-disjoint from $\bigcup_{j=0}^{i-1}E(T_j)$ with both ends in $\bigcup_{j=0}^{i-1}V(T_j)$; and
\item for each $i \in \{1,\dots, k\}$, the sequence $T_{1}, \ldots, T_{i}$ can be reordered into a subtrail of $C-x$.
\end{enumerate}
Note that we allow each $T_i$ to be a circuit. We call $T_1, \ldots, T_k$ the \emph{ears} of this decomposition.
We say that the ear decomposition is \emph{maximal} if there is no trail $T_{k+1}$ so that $H, T_1, \ldots, T_k, T_{k+1}$ is also an ear decomposition of $H$ into $C.$
The lemma below establishes that the notion of a maximal ear decomposition is dual to the closure.

\begin{lemma}
\label{lem:maximalEarDecomp}
Let $(G, \gamma)$ be a $4$-edge-connected Eulerian group-labeled graph, $x \in V(G)$ a vertex of $G$, $\C$ an $x$-packing saturating $x$, and $H$ an unused component of $\C$ of degree one adjacent to $C \in \C$ in $\mathsf{comp}(\C)$.
If $T_0:=H, T_1, \ldots, T_k$ is a maximal ear decomposition of $H$ into $C$, then $\bigcup_{j=0}^k V(T_j) = \mathrm{cl}(H)$.
\end{lemma}
\begin{proof}
We first show that $\bigcup_{j=0}^i V(T_j) \subseteq \mathrm{cl}(H)$ by induction on $i$; this holds for $i=0$ since $V(H) \subseteq \mathrm{cl}(H)$ by definition. Assuming $\bigcup_{j=0}^i V(T_j)\subseteq \mathrm{cl}(H)$, then since $T_{i+1}$ is a subtrail of $C-x$ which is disjoint from the previous ears, and $\delta(\mathrm{cl}(H))$ contains at most two edges from $C$, it must be that $V(T_{i+1})\subseteq \mathrm{cl}(H)$ as needed. Indeed otherwise $\delta(\mathrm{cl}(H))$ contains at least two edges in $E(T_{i+1})$ as well as the first and last edge of $E(C)$ with an end in $\mathrm{cl}(H)$.

Fix $U = V(H) \cup \bigcup_{i=1}^{k} V(T_i)$.  By the previous paragraph, $U \subseteq \mathrm{cl}(H)$.  Let $T'_1, \ldots, T'_{k}$ be an ordering of the sequence $T_{1}, \ldots, T_{k}$ such that $T' \coloneqq T'_1 \ldots T'_k$ is a subtrail of $C-x$.  If a vertex in $U$ occurs along $C-x$ before the beginning of $T'$, let $T_{k+1}$ be the subtrail of $C-x$ from the last such occurrence to the beginning of $T'$.  Then $T_0,T_1,\dots,T_{k+1}$ is an ear decomposition of $H$ into $C$ violating our choice of a maximal ear decomposition.  Similarly, if a vertex of $U$ occurs along $C-x$ strictly after the end of $T'$, we again get a contradiction to the maximality of our ear decomposition.  
We conclude that the occurrences of vertices of $U$ along $C$ form exactly the interval occupied by $T'$, and therefore $|\delta(U)\cap E(C)|=2$.
As $V(H) \subseteq U$, it follows that $ \mathrm{cl}(H) \subseteq U$, completing the proof.
%\PW{The original proof for $ \mathrm{cl}(H) \subseteq U$ was incorrect as written - chatgpt found an example where the argument fails.  Together, we implemented the above correction.}  
%Let $G'$ be the subgraph of $G$ containing only edges in $E(C) \cup E(H)$.
%We claim that every edge $e$ of $G'[\mathrm{cl}(H)]$ is in one of the trails $E(T_i)$, $i \in [k]$; this suffices to complete the proof since $G'[\mathrm{cl}(H)]$ is connected by minimality of the closure.
%Indeed, assume that there is an edge $e \in G'[\mathrm{cl}(H)]\setminus \bigcup_{j=0}^k E(T_j)$. Consider the unique maximal subtrail $T_{k + 1}$ of $C-x$ contained in $G'[\mathrm{cl}(H)] -  \bigcup_{j=0}^k E(T_j)$ and containing $e$; this exists by the definition of closure, and its endpoints lie in $\bigcup_{j = 0}^{k} V(T_{j})$. By maximality $T_{k+1}$, either $T_{k + 1}T'$ or $T'T_{k+1}$ is a subtrail of $C- x$. 
%This completes the proof as the ear decomposition is maximal.
\end{proof}

We note that while the above establishes that the underlying set of vertices in the maximal ear decomposition is unique, the number $k$ of ears or their order $T_1,\dots,T_k$ is not necessarily unique. However, this notion still gives rise to a meaningful way of measuring the complexity of unused components of degree one in an $x$-packing, which we now define. 

For a flexible unused component $H$ of $\C$ adjacent to $C \in \C$ in $\mathsf{comp}(\C)$, we let the \emph{rank of $H$ in $\C$}, denoted by $\mathrm{rank}_{\C}(H)$, be the minimum $i \in \mathbb{N}$ such that there is a maximal ear decomposition $T_0,T_1,\dots,T_k$ of $H$ into $C$ and some $C'\in \C$ different from $C$ with the property that $V(T_i)\cap V(C')\neq \emptyset$.  For flexible unused components Lemmas \ref{lem:rankwelldfn} and \ref{lem:maximalEarDecomp} imply that such $i$ always exists, and also by definition $\rank_\C(H)\geq 1$.
For an inflexible unused component $H$ of $\C$ we define $\mathrm{rank}_{\C}(H) = \infty.$
Furthermore, we define the \textit{rank of $\C$}, denoted by $\mathrm{rank}(\C)$, to be the minimum of $\mathrm{rank}_{\C}(H)$ over all unused components $H$ of $\C$ of degree one if there is at least one such component, and $0$ otherwise.

We are now ready to proceed with the proof of the main theorem of this section which we repeat here for convenience.

\decomptocircuits*

To briefly summarize the argument, under the theorem's assumptions, we show there is an $x$-packing $\C$ with the property that every circuit has degree at most one and every unused component has degree at least two in $\mathsf{comp}(\C)$.
Consequently, this leads to a small number of unused components which can be greedily absorbed in the packing.

\begin{proof}[Proof of Theorem \ref{thm:EP}]
    We choose an $x$-packing $\C$ in $(G,\gamma)$ subject to the following conditions, satisfied in the order listed below:
    \begin{enumerate}%[label=\Alph*.]
        \item\label{cond:1} $\C$ contains exactly $\frac{\deg(x)}{2}$ many circuits and at least $k$ of these have labels outside of $\Gamma'$;
        \item\label{cond:2}  the number of edges covered by $\C$ is maximized;
        \item\label{cond:3}  the number of unused components of $\C$ of degree at least two is maximized; 
        \item\label{cond:4}  the rank of $\C$ is minimized. 
    \end{enumerate}
In general, any $x$-circuit $C$ with group value not in $\Gamma'$ which visits $x$ internally can be decomposed into two distinct $x$-circuits, at least one of which has group value not in $\Gamma'$.  Along with the fact that $G$ is Eulerian, we conclude that there exists an $x$-packing $\C$ which satisfies property \ref{cond:1}~and, consequently, properties \ref{cond:1}-\ref{cond:4}.
The assumptions of Theorem \ref{thm:EP} imply that there exists such a collection; we proceed to prove a few claims about its structure. 
        \begin{claimth}
        $\deg_\C(C)= 0$ for all circuits $C \in \C$ with $\gamma(C)\in \Gamma'$. 
    \end{claimth}
    \begin{claimproof}
        Let $C \in \C$ satisfy $\gamma(C)\in \Gamma'$ and assume for a contradiction that there is some unused component $H$ which shares a vertex with $C$. Let $C'$ be the circuit obtained by joining an Eulerian tour of $H$ to $C$ at a common vertex, and $\C'$ the collection obtained from $\C$ by swapping $C$ with $C'$. It follows that $\C'$ still satisfies condition \ref{cond:1} while the number of edges covered by $\C'$ is now strictly larger; contradiction. It follows that $\deg_\C(C)=0$ for every $C \in \C$ with $\gamma(C)\in \Gamma'$. 
    \end{claimproof}
    
    \begin{claimth}
        $\deg_\C(C)\leq 1$ for all circuits $C \in \C$ with $\gamma(C)\not\in \Gamma'$. 
    \end{claimth}
    \begin{claimproof}
        Fix some $C \in \C$ with $\gamma(C)\not\in\Gamma'$, and assume for a contradiction that $\deg_\C(C)\geq 2$, i.e. there are two distinct unused components $H_1$ and $H_2$ which both share vertices with $C$.
        Suppose
%\CM{this is a slight abuse of notation since a trail is defined as a sequence of oriented edges.}\EP{Rewrote with edges.}
        \[ C = \vec{e}_1 \dots \vec{e}_p \vec{f}_1 \vec{e}_{p+1} \dots \vec{e}_{p+\ell} \vec{f}_2 \vec{e}_{p+\ell+1} \dots \vec{e}_{p+\ell+m};\]
        where the tail of $\vec{e}_1$ is $x$, the tail of $\vec{f}_1$ is $h_1$ and $h_1 \in V(C)\cap V(H_1)$, and the tail of $\vec{f}_2$ is $h_2$ and $h_2 \in V(C)\cap V(H_2)$.
        Let $\vec{H_1}$ be an Eulerian circuit in $H_1$ starting and ending at $h_1$, and likewise define $\vec{H_2}$.
        Consider the following elements of $\Gamma$
        \[ a:=\gamma(\vec{e}_1 \dots \vec{e}_p),\quad b:=\gamma(\vec{f}_1 \vec{e}_{p+1} \dots \vec{e}_{p+\ell}), \]
        \[ c:=\gamma(\vec{f}_2 \vec{e}_{p+\ell+1} \dots \vec{e}_{p+\ell+m}),\quad \chi_1:=\gamma(\vec{H_1}), \quad \chi_2:=\gamma(\vec{H_2}), \]
        and observe that
        \[ (ab\chi_2c)\cdot(a\chi_1b\chi_2c)^{-1}\cdot(a\chi_1bc)=abc = \gamma(C) \not\in\Gamma'.\]
        It follows that one of $a\chi_1bc$, $ab\chi_2c$, and $a\chi_1b\chi_2c$ must lie outside $\Gamma'$.
        We may therefore appropriately glue $H_1$, $H_2$, or both at $C$ to obtain a strictly bigger circuit $C'$ with $\gamma(C')\not\in\Gamma'$.
        By swapping $C$ and $C'$ we obtain a new collection of circuits covering strictly more edges than before that still satisfies condition \ref{cond:1}, giving rise to a contradiction once again.
    \end{claimproof}
    \begin{claimth}
        $\deg_\C(H)\geq 2$ for all unused components $H$ of $\C$.
    \end{claimth}
    \begin{claimproof}
        By connectivity of $G$, every unused component must have degree at least one.  Assume for a contradiction that there is an unused component of degree one.
        First note that by the two claims above any unused component of degree one is flexible.
        Let $H$ be such a component with the property that $n:=\rank_\C(H)=\rank(\C) \geq 1$, i.e. its rank is minimal over all unused components of degree one.
        Let $C \in \C$ be the unique circuit of $\C$ adjacent to $H$ in $\mathsf{comp}(\C)$.
        By the first claim, we know that $\gamma(C)\not\in\Gamma'$.
        Moreover, fix a maximal ear decomposition $T_0 \coloneqq H, T_{1}, \ldots, T_k$ of $H$ into $C$ and $C'\in \C$ distinct from $C$ with the property that $V(T_{n})\cap V(C')\neq \emptyset$.
        We shall construct a new $x$-packing $\tilde \C$ with the property that, either the number of unused components of degree at least two is strictly larger than that of $\C$, or it has the same number of unused components of degree at least two as $\C$ but its rank is strictly smaller than that of $\C$; both of these lead to a contradiction to the choice of $\C$.

        Towards this goal, consider any vertex $v \in V(T_{n}) \cap V(C').$
        First note that $v$ cannot be a vertex in $\bigcup_{i=0}^{n-1} V(T_i)$ as this would contradict the minimality of $n$; in particular, $v$ cannot be an endpoint of $T_n$. Choose a pair of subtrails $T^1_{n}$ and $T^2_{n}$ by arbitrarily ``breaking'' $T_n$ at $v$, i.e. non-empty subtrails of $T_{n}$ having $v$ as head and tail, respectively, and satisfying $T_{n} = T^{1}_{n} T^{2}_{n}$.
        Since $T_{n}$ is a subtrail of $C - x$, we may also rewrite $C$ as $C = C_{1} C_{2},$ where $C_1$ is the subtrail from $x$ to $v$ containing $T^1_n$ and $C_2$ is a subtrail from $v$ to $x$ containing $T^2_n$.  
      Similarly, we let $C' = R_{1}R_{2}$ where $R_1$ is a subtrail of $C'$ from $x$ to $v$ and $R_2$ is a subtrail of $C'$ from $v$ to $x$.
        
        Next, we claim that at least one of the sets $\{ C_{1}R_{2}, R_{1}C_{2} \}$ or $\{ C_{1}R^{-1}_{1}, C^{-1}_{2}R_{2} \}$ has at least as many $x$-circuits labeled outside $\Gamma'$ as the collection $\{ C, C' \}.$
        
        Consider the following elements of $\Gamma$ 
        $$a \coloneqq \gamma(C_{1}), \quad b \coloneqq \gamma(C_{2}), \quad c \coloneqq \gamma(R_{1}), \quad  d \coloneqq \gamma(R_{2}),$$ 
        and observe that 
        $$ac^{-1} \cdot cb = ab = \gamma(C) \not\in \Gamma', \quad b^{-1}d \cdot (ad)^{-1} = (ab)^{-1} = \gamma(C)^{-1} \not\in \Gamma',$$ 
        and that 
        $$cb \cdot b^{-1}d = cd = \gamma(C'), \quad (ad)^{-1} \cdot ac^{-1} = (cd)^{-1} = \gamma(C')^{- 1}.$$
        Since $\gamma(C) \not\in \Gamma'$ it follows that either $ac^{-1} = \gamma(C_{1}R^{-1}_{1}) \not\in \Gamma'$ or $cb = \gamma(R_{1}C_{2}) \not\in \Gamma'.$
        Hence, if $\gamma(C') \in \Gamma'$, the claimed property holds for one of the two sets $\{ C_{1}R_{2}, R_{1}C_{2} \}$ or $\{ C_{1}R^{-1}_{1}, C^{-1}_{2}R_{2} \}$. On the other hand, if
        $\gamma(C') \not\in \Gamma'$ then at least one of the two factors in all four equations above must be outside $\Gamma'$ and again the desired property holds through a simple checking of the cases.  

        Let us assume without loss of generality that the claim above holds for the set $\{ C_{1}R_{2}, R_{1}C_{2} \}.$
        We complete the proof by showing that the collection $\tilde{\C}$ of $x$-circuits obtained from $\C$ by replacing circuits $C$ and $C'$ with circuits $C_{1}R_{2}$ and $R_{1}C_{2}$ yields a contradiction to the choice of $\C$.

        We first verify that the collection $\tilde{\C}$ is at least as good as $\C$ with respect to conditions \ref{cond:1}, \ref{cond:2}, and \ref{cond:3}.
        Condition \ref{cond:1} is clearly satisfied given the previous claim.
        Moreover, the edges covered by $\tilde{\C}$ are by definition precisely the same as the edges covered by $\C$, and so condition \ref{cond:2} is also satisfied. 
        Consider an unused component $K$ of $\C$ of degree at least two.  Any adjacency to a circuit other than $C$ and $C'$ is also present in $\mathsf{comp}(\tilde{\C})$, and $K$ cannot meet $C$ since otherwise $K=H$.  If $K$ meets $C'$, then $K$ also meets at least one of $C_{1}R_{2}$ or  $R_{1}C_{2}$.  We conclude $K$ is of degree at least two in $\mathsf{comp}(\tilde{\C})$ and condition \ref{cond:3} holds for $\tilde{\C}$.

        There are now two cases to distinguish.

        \smallskip\noindent
        \textbf{Both $C_1R_2$ and $R_1C_2$ intersect $H$:} By definition, this means that the degree of $H$ in $\mathsf{comp}(\tilde{\C})$ is two.
        Combined with the observations above, this implies that $\tilde{\C}$ strictly improves on condition~\ref{cond:3} over $\C$ which gives us the desired contradiction.

        \smallskip\noindent
        \textbf{Only one of $C_1R_2$ and $R_1C_2$ intersects $H$:} By assumption, $H$ has degree exactly one in $\mathsf{comp}(\tilde \C)$.  Since $\tilde{\C}$ satisfies conditions \ref{cond:1}-\ref{cond:3}, the first two claims hold for $\tilde{\C}$ as well and we see that the unique circuit meeting $H$ has degree one in $\mathsf{comp}(\tilde \C)$ and label not in $\Gamma'$.  Thus, $H$ is flexible and its rank is finite.  In this case it must be that $n=\rank_\C(H)$ is at least two. 
        Indeed, if $n = 1,$ then by definition of ear decompositions, the endpoint of each $T^{i}_{n}$ for $i \in [2]$ that is not $v$ is a vertex of $H$.
        Therefore by construction both $C_{1}R_{2}$ and $R_{1}C_{2}$ would intersect $H,$ contrary to our assumption. 
        
        Now, without loss of generality assume that $V(C_{1}R_{2}) \cap V(H) \neq \emptyset$. Let $T^*$ be the subtrail of $C-x$ obtained from concatenating the ears $T_{1}, \ldots, T_{n-1}$ in the appropriate order.
        Note that, since $n \geq 2,$ $T^*$ is non-empty.
        Additionally, $v \not\in V(T^*)$ by our assumptions, implying in particular that $T^*$ is either a subtrail of $C_1 - x$, or a subtrail of $C_2 -x$. Moreover, $V(T^*) \cap V(H) \neq \emptyset$ while $V(R_1C_2)\cap V(H) =\emptyset$, implying that $T^*$ is in fact a subtrail of $C_1-x$, and so in particular a subtrail of $C_{1}R_{2} - x$. As such, $H,T_1,\dots,T_{n-1}$ is a valid ear decomposition of $H$ into $C_1R_2$. 
        Let $w$ be the endpoint of $T^2_n$ other than $v$. By the definition of an ear, $w \in V(H)\cup V(T^*)$. Since $ w \in V(C_2)$ and $V(C_2) \cap V(H)=\emptyset$, we have $w \in V(T^*)$.  Given that every ear decomposition can be trivially extended to a maximal ear decomposition, it follows that the $x$-circuit $R_1C_2 \in \tilde \C$ witnesses that the rank of $H$ in $\tilde \C$ is at most $n-1$. In particular, the rank of $\tilde \C$ is strictly smaller than that of $\C$, a contradiction completing the proof of the claim.
    \end{claimproof}
    
    Let $q$ be the number of circuits of the packing with label not in $\Gamma'$, and let $r$ be the number of unused components. 
    The claims imply $2r \le q$. Choose for each unused component one incident circuit; these $r$ chosen circuits are distinct and all have labels not in $\Gamma'$.  Insert an Euler tour of each unused component into its assigned circuit. Concatenate the resulting $r$ circuits and all circuits of the original packing labeled in $\Gamma'$, at $x$, to form a single circuit $D$, provided the collection is non-empty. The other $q-r$ circuits together with $D$ cover every edge of $G$. If $D$ is absent, the unchanged circuits already form the required decomposition. If $D$ is present and $\gamma(D)$ is not in $\Gamma'$, keep the circuit $D$. Otherwise concatenate $D$ at $x$ with one of the $q-r$ unchanged circuits; the resulting circuit has label not in $\Gamma'$. In every case, the resulting decomposition has at least $q-r \ge \lceil q/2 \rceil \ge \lceil k/2\rceil$ circuits.
    \end{proof}
    
\section{Making an immersion near-full}\label{sec:NearFull}

We first discuss the ``universal'' group-labeled graphs which we look for as (near-full) immersions.

For positive integers $k$ and $n$, we define the \emph{$(k, n)$-flower} to be the graph on $n + 1$ vertices $x,y_1,\dots,y_n$ so that for each $i \in [n]$, there are $k$ parallel edges between $x$ and $y_i$.
We call $x$ the \emph{center vertex} of the flower.
If $\phi$ is an immersion of a flower into a graph $G$, then we call the branch vertex of $\phi$ corresponding to the center of the flower the \emph{branch center}.
A \emph{petal} is a subgraph of the flower consisting of $x$, some $y_i$, and the $k$ parallel edges between them.
Note that if $H$ is a graph on $n$ vertices where each vertex has degree at most $k$, then the $(k, n)$-flower admits $H$ as an immersion; we map $V(H)$ to $\{y_1,y_2,\dots,y_n\}$, and then we greedily map $E(H)$ to two-edge walks of the flower.

Let $\Gamma$ be a finite group.
The \emph{$(\Gamma,k,n)$-rich flower} is the $\Gamma$-labeled graph where the graph is the $(|\Gamma|k,n)$-flower and, for all $i\in[n]$ and $\alpha\in\Gamma$, there exist $k$ oriented edges from $x$ to $y_i$ with label $\alpha$.
Note that if $(H,\gamma_H)$ is a $\Gamma$-labeled graph on $n$ vertices where each vertex has degree at most $k$, then the $(\Gamma,k,n)$-rich flower admits $(H,\gamma_H)$ as an immersion.
To see this, first map $V(H)$ to $\{y_1,y_2,\dots,y_n\}$ as before.
Then select one arbitrary orientation $\vec{e}$ of each edge $e$ of $H$, and map $\vec{e}$ to a two-edge trail which begins with an edge labeled by the identity and oriented towards the center, and ends with a $\gamma_H(\vec{e})$-labeled edge oriented away from the center.
%\PW{it was necessary to switch from "at most k edges" to "deg at most k" to deal with loops.

The goal of this section is to prove the following theorem which essentially allows us to reduce the question of finding large rich flowers as near-full immersions in $4$-edge-connected Eulerian group-labeled graphs to finding (even larger) rich flowers as immersions.

\begin{restatable}{thm}{flowerReduction}
\label{thm:flowerReduction}
    Let $\Gamma$ be a finite group and $k$ and $n$ be positive integers. Let $k'$ and $n'$ be integers such that $k' \geq k|\Gamma|^{4 + \log_2|\Gamma |} + 1$ and $n' \geq n|\Gamma|$. Suppose $(G,\gamma)$ is a 4-edge-connected Eulerian $\Gamma$-labeled graph such that there exists an immersion $\phi'$ of a $(\Gamma,k',n')$-rich flower with center $x'$ in $(G,\gamma)$.  Then there exists an $x$-near-full immersion $\phi$ of a $(\Gamma,k,n)$-rich 
  flower with center $x$ in $(G,\gamma)$.  Moreover, if $X$ and $X'$ are the sets of branch vertices of $\phi$ and $\phi'$, respectively, then $X \subseteq X'$ and, specifically, $\phi'(x') = \phi(x)$.  
\end{restatable}

The main technical difficulty is the following lemma, which is used to ``disentangle'' an $(x, \Gamma')$-decomposition from a flower immersion. 

\begin{lemma}\label{lem:uncrossing}
    Fix an integer $k \geq 2$.
    Let $\Gamma$ be a finite group, let $\Gamma'$ be a proper subgroup of $\Gamma$, and let $(G,\gamma)$ be a 4-edge-connected Eulerian $\Gamma$-labeled graph.
    Suppose that $(G,\gamma)$ contains an $x$-near-full immersion $\varphi$ of a $\Gamma'$-labeled Eulerian graph $(H, \gamma_H)$ where $H$ has every edge incident to a single vertex $x$.  Let $\gamma_G'$ be an accommodating labeling of the immersion.  
    Furthermore, suppose that $(G,\gamma_G')$ contains an $(\varphi(x),\Gamma')$-decomposition of size at least $k$.
    Then $(G,\gamma)$ has an $x$-near-full immersion of a graph obtained from $(H,\gamma_H)$ by deleting at most $2k+2$ edges and adding $k$ loops at $x$ with label outside $\Gamma'$.
\end{lemma}

\begin{proof}
    For the duration of the proof, labels refer to the fixed shifting $\gamma_G'$ of $\gamma$.  We first extend the $x$-near-full immersion $\varphi$ to a full immersion of the graph obtained from $(H, \gamma_H)$ by possibly adding a single loop $\hat{e}$ at $x$ with the appropriate $\gamma_G'$-label (if the $x$-near-full is in fact a full immersion, then we do not add $\hat{e}$). For notational convenience we will assume $x = \varphi(x)$. We partition $E(G)$ into circuits $\mathcal{C} = \{C_1, \dots, C_\ell\}$ each beginning and ending at $x$ such that
    \begin{enumerate}
        \item at least $k$ of them have label outside $\Gamma'$, and
        \item the size of the symmetric difference of transitions of $\varphi$ and transitions of $\mathcal{C}$ is minimized.
    \end{enumerate}
     We note that we minimize transitions with respect to $\varphi$ only after extending it to a full immersion. Some choice of $\mathcal{C}$ exists because $(G,\gamma_G')$ contains an $(x, \Gamma')$-decomposition of size at least $k$. Order the circuits such that $C_1, \dots, C_m$ have label outside $\Gamma'$, and $C_{m+1}, \dots, C_\ell$ have label in $\Gamma'$ for some $m \geq k$. We partition $E(G)$ into $E_1 = \bigcup_{i=1}^m E(C_i)$ and $E_2 = \bigcup_{i=m+1}^\ell E(C_i)$.

    We claim that for any branch trail $T$ of $\varphi$, either every edge of $T$ is contained in $E_1$ or every edge of $T$ is contained in $E_2$. Suppose otherwise, and let $\vec{e}, \vec{f}$ be a pair of edges in $T$ such that $(\vec{e}, \vec{f})$ is a transition of $T$ and $e$ is in a circuit $C_i$, $f$ is in a circuit $C_j$, and exactly one of $i,j$ is in $[m]$. By possibly reversing the orientation of $T$ we may assume $i \in [m]$ and $j \in [\ell] \setminus [m]$, and similarly we may orient $C_i, C_j$ such that $\vec{e}$ is in $C_i$ and $\vec{f}$ is in $C_j$. Let $C_i = C^1_iC^2_i$ such that $\vec{e}$ is the last edge of $C^1_i$, and let $C_j = C^1_jC^2_j$ such that $\vec{f}$ is the first edge of $C^2_j$, where possibly $C^2_i$ or $C^1_j$ is empty.   If both are empty, consider removing $C_i$ and  $C_j$ from $\mathcal{C}$ and adding $C^1_iC^2_j = C_iC_j$. This circuit has label $\gamma_G'(C_i)\gamma_G'(C_j) \not\in \Gamma'$, so the same number of circuits in the new collection are labeled outside $\Gamma'$. Furthermore this adds the transition $(\vec{e}, \vec{f})$ to the collection, contradicting the choice of $\mathcal{C}$.
    
    Hence we may assume one of $C^2_i, C^1_j$ is nonempty. Consider replacing $C_i, C_j$ in $\mathcal{C}$ by the circuits $C^1_iC^2_j$, $C^1_jC^2_i$. Note that this adds the transition $(\vec{e}, \vec{f})$ to the collection while only removing transitions which do not appear in $\varphi$. Thus $C^1_iC^2_j$, $C^1_jC^2_i$ is a better choice of circuits unless $\gamma_G'(C^1_iC^2_j), \gamma_G'(C^1_jC^2_i) \in \Gamma'$. However because $\gamma_G'(C_j) \in \Gamma'$ this would imply
    \begin{align*}
        \gamma_G'(C_i) &= \gamma_G'(C^1_i)\gamma_G'(C^2_i)\\
        &= \gamma_G'(C^1_i)\gamma_G'(C^2_j)\gamma_G'(C^2_j)^{-1}\gamma_G'(C^1_j)^{-1}\gamma_G'(C^1_j)\gamma_G'(C^2_i)\\
        &= \gamma_G'(C^1_iC^2_j)\gamma_G'(C_j)^{-1}\gamma_G'(C^1_jC^2_i)\\
        &\in \Gamma',
    \end{align*}
    a contradiction. This completes the proof that every trail of $\varphi$ has all edges entirely contained in one of $E_1, E_2$.

We now claim that there are at most $k+1$ circuits with label outside
$\Gamma'$ in $\mathcal C$, that is, $m\leq k+1$.
Suppose for a contradiction that $m\geq k+2$.
First, we observe that every branch trail of $\varphi$ has all its edges in a single
circuit of $\mathcal C$.
If not, some branch trail has a transition $(\vec e,\vec f)$
whose edges belong to distinct circuits $C_i$ and $C_j$.
Reconnecting these circuits as above adds this transition and removes
only transitions not belonging to $\varphi$.
The resulting collection still contains at least $m-2\geq k$ circuits
with label outside $\Gamma'$, contradicting the choice of $\mathcal C$.

Choose a circuit $C_i$ with label outside $\Gamma'$ which does not
contain any edge of $\varphi(\hat e)$, if $\hat e$ is defined.
Let $J$ be the subgraph of $H$ consisting of the edges whose branch
trails have their edges in $C_i$.
Since the extended immersion is full, these branch trails partition
$E(C_i)$. By construction, every vertex of $J$ has even degree.  Moreover, every edge of $J$ is incident with $x
$, so we conclude that $J$ is connected and has an Eulerian tour beginning and ending at $x$.

Following this Eulerian tour through the corresponding branch trails
gives a circuit $D$ beginning and ending at $x$ with $E(D) = E(C_i)$ but possibly changing the order and orientation of the edges of $C_i$.  
Every edge of $J$ has label in $\Gamma'$, and hence
$\gamma_G'(D)\in\Gamma'$.
Furthermore, $D$ contains every transition of $\varphi$ on these
edges.
The circuit $C_i$ must omit at least one such transition, as otherwise
it would itself be a concatenation of complete branch trails,
possibly reversed, and would therefore have label in $\Gamma'$, which is not the case.  
Since $C_i$ and $D$ have the same number of transitions, it follows that replacing
$C_i$ by $D$ strictly decreases the symmetric difference with the
transitions of $\varphi$.
This replacement leaves $m-1\geq k$ circuits with label outside
$\Gamma'$, again contradicting the choice of $\mathcal C$.
We conclude that $m\leq k+1$.

Hence there are at most $k+1$ circuits in $\mathcal{C}$ with label outside $\Gamma'$. Let $H'$ be the subgraph of $H$ which includes only edges whose corresponding trails are contained in $E_2$. We claim that $H'$ is formed from $H$ by deleting at most $2k+2$ edges. Let $E'$ be the set of edges incident to $x$, including loops, which appear as the first or last edge of a branch trail of $\varphi$ beginning or ending at $x$, respectively.  We will show that for any circuit $C_i$ with $i \in [m]$, $C_i$ contains at most two edges in $E'$. This implies the claim as follows. Note that as every edge in $H$ is incident to $x$, every branch trail of $H$ contains some edge in $E'$. Then an edge $e$ is in $E(H) \setminus E(H')$ if and only if $\varphi(e)$ contains an edge of $E_1 \cap E'$. Thus $|E(H)\setminus E(H')|\leq |E_1 \cap E'| \leq 2m \leq 2k+2$, proving the claim.

Suppose to reach a contradiction that a circuit $C_i$ for $i \in [m]$ contains at least three edges in $E'$. 
Choose one of these edges which is neither the first nor the last edge of $C_i$, and consider its end at $x$ where its branch trail begins or ends. This end occurs at an internal visit of $C_i$ to $x$, giving a transition $(\vec e,\vec f)$ with the head of $\vec e$ and the tail of $\vec f$ equal to $x$. Since the chosen end is an end of a branch trail, this transition does not belong to $\varphi$.
Let $C_i = C_i^1C_i^2$ where the last edge of $C_i^1$ is $\vec{e}$ and the first edge of $C_i^2$ is $\vec{f}$. Note that $\gamma_G'(C_i) = \gamma_G'(C_i^1)\gamma_G'(C_i^2)$, so one of these circuits has label outside $\Gamma'$. Then we can modify $\mathcal{C}$ by replacing $C_i$ with the two circuits $C_i^1, C_i^2$ to contradict the choice of $\mathcal{C}$ to minimize the symmetric difference with the transitions of $\varphi$.

We then find the desired $x$-near-full immersion as follows. On $E_1$ we take the circuits $C_1, \dots, C_k$ as immersions of loops, and on $E_2$ we take $\varphi$ restricted to $H'$. This accounts for all edges except possibly edges of $C_{k+1}$ if $m = k+1$ and  
edges of $\varphi(\hat{e})$ if it exists and its branch trail lies in $E_2$.  Each such subset of edges, if it is non-empty, forms a circuit through $x$, and we conclude that the $x$-near-full immersion exists as claimed, completing the proof.  
\end{proof}

The near-full conclusion in the statement of the lemma is necessary in general.  Already 
for $\Gamma=\mathbb Z_2$ and $\Gamma'=\{0\}$, the total label of $G$ determines the parity of $m$, and may force $m=k+1$ rather than $m=k$.

%\CMin{perhaps say something here about how we need to have a near-full because we can not choose the value of $m \in \{k, k+1\}$ due to parity constraints.}

We need one more type of flower in order to build up to rich flowers.
We say that a $\Gamma$-labeled graph is a \emph{$(\Gamma,k,n)$-generating flower} if there exists a generating set $S$ for $\Gamma$ so that $1 \in S$, the graph is an $(|S|k, n)$-flower, and for all $i \in [n]$ and $\alpha \in S$, there exist $k$ oriented edges from $x$ to $y_i$ with label $\alpha$.
We call $S$ the \emph{generator}.

Our main strategy will be to start with a flower with every edge labeled by the identity, and grow generating flowers for larger and larger subgroups until we have a rich flower. Hence we need the following two lemmas, which give the starting and ending point of the process.

The first is a Ramsey-type lemma for flowers in group-labeled graphs.

\begin{lemma}[{\cite{McCartyMW2025Structure}}]\label{lem:ZeroFlower}
Let $\Gamma$ be a finite group, let $k$ and $n$ be positive integers, and let $(G, \gamma)$ be a $\Gamma$-labeled graph so that $G$ is a $(k|\Gamma|, n)$-flower with center $x$. Then there is an $x$-near-full immersion of the $(\{1\}, k, n)$-rich flower with center $x$ into $(G,\gamma)$ and the peripheral branch vertices are precisely the peripheral vertices of $G$.
\end{lemma}

The second lemma shows that for large enough $k$, every $(\Gamma,k,n)$-generating flower is ``universal''.

\begin{lemma}[{\cite{McCartyMW2025Structure}}]\label{lem:GenFlowerUniversal}
Let $\Gamma$ be a finite group and $k$ and $n$ be positive integers. Then any $(\Gamma, k|\Gamma|^2, n)$-generating flower $G$ with center $x$ admits an $x$-near-full immersion of the $(\Gamma, k, n)$-rich flower with center $x$ mapping its peripheral branch vertices bijectively to the peripheral vertices of $G$.
\end{lemma}

The above lemmas, as originally stated in \cite{McCartyMW2025Structure} as lemmas 4.2 and 4.1, respectively, do not specify the branch vertices beyond the center vertex of the star and do not specify that the immersion is $x$-nearly full.  The stronger statements above follow immediately from the proofs in \cite{McCartyMW2025Structure} and we omit the derivation here.

We are now ready to prove the main result which we repeat here for convenience.

\flowerReduction*

\begin{proof}
By possibly decreasing $k'$ by one, we assume $k'$ is even with $k' \geq k|\Gamma|^{4 + \log_2|\Gamma|}$.
Let $x$ be the branch center of the $(\Gamma, k', n')$-rich flower immersion. Consider arbitrarily extending this to a full immersion, so $(G,\gamma)$ has a full immersion of the $(|\Gamma|k', n')$-flower with branch center $x$ and arbitrary labels. By Lemma~\ref{lem:ZeroFlower}, this immerses the $(\{1\}, k', n')$-rich flower with branch center $x$. Because every edge of the rich flower is incident to $x$, this is in fact an $x$-near-full immersion. 
    
Thus by Observation~\ref{obs:transitivity of almost total immersions}, $(G,\gamma)$ has an $x$-near-full immersion of the $(\{1\}, k', n')$-rich flower with branch center $x$. Note also that for any proper subgroup $\Gamma'$ of $\Gamma$, and any shifting $\tilde{\gamma}$ of $\gamma$, $(G,\tilde{\gamma})$ has an $(x, \Gamma')$-packing of size $k'n'$.  This holds because the original rich flower has such a packing for every proper subgroup, and shifting simultaneously conjugates the labels of all circuits based at $x$.

We now find a $(\Gamma', \hat{k}, \hat{n})$-generating flower as an $x$-near-full immersion $\varphi$ in $(G,\gamma)$ whose branch vertices belong to $X'$ and whose center maps to $x=\phi'(x')$, such that
\begin{enumerate}
        \item $\Gamma'$ is a subgroup of $\Gamma$,
        \item $\hat{k}$ is even with $k|\Gamma|^2 \cdot (4|\Gamma|)^{\log_2(|\Gamma|/|\Gamma'|)} \leq \hat{k} \leq k'$,
        \item $n\frac{|\Gamma|}{|\Gamma'|} \leq \hat{n} \leq n'$, and
        \item subject to the above, $|\Gamma'|$ is as large as possible.
\end{enumerate}
Note that some choice exists by taking the $(\{1\}, k', n')$-rich flower.
We now claim that such a choice must have $\Gamma' = \Gamma$.
If this holds, then we have a $(\Gamma, k|\Gamma|^2, n)$-generating flower as an $x$-near-full immersion.
We then apply Lemma~\ref{lem:GenFlowerUniversal} and Observation~\ref{obs:transitivity of almost total immersions} to obtain a $(\Gamma, k, n)$-rich flower as an $x$-near-full immersion in $(G,\gamma)$, with branch vertices contained in $X'$ and branch center $\phi'(x')$, as required.

Suppose to reach a contradiction that $\varphi$ is an $x$-near-full immersion of a $(\Gamma', \hat{k}, \hat{n})$-generating flower $(H, \gamma_H)$ for $\Gamma'$ a proper subgroup of $\Gamma$.
Note that because $\hat{k}$ is even, $H$ is Eulerian.
Let $k^*$ be an even integer such that $\frac{\hat{k}}{4|\Gamma|} \leq k^* \leq \frac{\hat{k} - 2}{2|\Gamma|}$, and let $n^* = \lfloor \hat{n}/2\rfloor$.
Note that because $k \geq 1$ and $|\Gamma| \geq |\Gamma|/|\Gamma'| \geq 2$, we have $\hat{k} \geq 4|\Gamma|^3 \geq 8|\Gamma| + 4$.
Hence $\frac{\hat{k} - 2}{2|\Gamma|} \geq \frac{\hat{k}}{4|\Gamma|} + 2$, so such a choice of $k^*$ exists.
    
Let $\gamma_G'$ be an accommodating labeling for $\varphi$. By the preceding observation, $(G,\gamma_G')$ contains an $(x,\Gamma')$-packing of size $k'n' \geq \hat{k}\hat{n} \geq 2(|\Gamma|-1)k^*n^*$. By Theorem~\ref{thm:EP}, $(G,\gamma_G')$ therefore contains an $(x,\Gamma')$-decomposition of size at least $(|\Gamma|-1)k^*n^*$.
By Lemma~\ref{lem:uncrossing}, $(G,\gamma)$ contains an $x$-near-full immersion of a graph obtained from $(H,\gamma_H)$ by deleting at most $2(|\Gamma| - 1)k^*n^* + 2$ edges and adding $(|\Gamma| - 1)k^*n^*$ loops at $x$ with label outside $\Gamma'$. We then delete the $n^*$ petals which had the most edges deleted.  If $\hat{n}$ is odd, we discard one additional petal so that exactly $n^*$ petals remain.  Each remaining petal had at most $2(|\Gamma| - 1)k^* + 2$ edges deleted. By possibly ignoring some branch trails of the immersion (and utilizing Observation~\ref{obs:transitivity of almost total immersions}), we see that $(G,\gamma)$ has an $x$-near-full immersion of the graph obtained from the $(\Gamma', 2k^*, n^*)$-generating flower by adding $(|\Gamma| - 1)k^*n^*$ loops at $x$ with label outside $\Gamma'$. 
We then take $k^*n^*$ of these loops which have the same label $\alpha \in \Gamma \setminus \Gamma'$. Let $\Gamma''$ be the subgroup generated by $\Gamma' \cup \{\alpha\}$. By sacrificing $k^*$ trails with identity label on each petal, we obtain an $x$-near-full immersion of the $(\Gamma'', k^*,n^*)$-generating flower with branch center $x$. Because $\Gamma'$ is a proper subgroup of $\Gamma''$, $|\Gamma''|/|\Gamma'|$ is an integer, and hence at least 2. Thus
$$k^* \geq \frac{\hat{k}}{4|\Gamma|} \geq k|\Gamma|^2 \cdot (4|\Gamma|)^{\log_2(|\Gamma|/|\Gamma'|) - 1} \geq k|\Gamma|^2 \cdot (4|\Gamma|)^{\log_2(|\Gamma|/|\Gamma''|)}$$
$$n^* = \left\lfloor \frac{\hat{n}}{2}\right\rfloor \geq \left\lfloor \frac{|\Gamma|}{|\Gamma'|} \frac{n}{2} \right\rfloor \geq \left\lfloor \frac{2|\Gamma|}{|\Gamma''|} \frac{n}{2} \right\rfloor = \frac{|\Gamma|}{|\Gamma''|}n.$$
Therefore this immersion contradicts the choice of $\varphi$.
\end{proof}

Theorem \ref{thm:mainReduction}, which we restate for convenience, now follows easily.

\mainReduction*

\begin{proof}
Fix a vertex $x \in V(H)$.  Let $k = 2|E(H)|+2$ and $n = |V(H)|$.  Let
$(F,\gamma_F)$ be the $(\Gamma,k,n)$-rich flower.

There is an immersion $\psi$ of $(H,\gamma_H)$ into $(F,\gamma_F)$ mapping $x$ to the center of $F$ such that every unused edge of $F$ is incident with its center, making $\psi$ an $x$-near-full immersion.  

Choose $k'$ and $n'$ satisfying the bounds in Theorem~\ref{thm:flowerReduction}, and let $(H',\gamma_{H'})$ be the $(\Gamma,k',n')$-rich flower.  We claim $(H', \gamma_{H'})$ is the desired $\Gamma$-labeled graph.  Suppose that $(G,\gamma)$ is Eulerian and $4$-edge-connected and contains $(H',\gamma_{H'})$ as an immersion.  By Theorem~\ref{thm:flowerReduction}, $(G,\gamma)$ contains $(F,\gamma_F)$ as a near-full immersion at its center.  Composing this immersion with $\psi$ and applying Observation~\ref{obs:transitivity of almost total immersions} gives an $x$-near-full immersion of $(H,\gamma_H)$ in $(G,\gamma)$, as required.
\end{proof}

%%%%%%%%%%%%%%%%%%%%%%%%
%%%%%%%%%%%%%%%%%%%%%%%%%

\section{Extending a near-full immersion}\label{sec:ExtendingNear}

In this section we discuss under what conditions the near-full immersion in Theorem~\ref{thm:mainReduction} can be replaced by a full immersion. This cannot always be the case: suppose that $(G,\gamma_G)$ and $(H,\gamma_H)$ are Eulerian graphs labeled over $\mathbb{Z}^k_2$ for some $k$ such that the sum of all edge weights in $G$ is different from the sum of all edge weights in $H$ (note that orientation doesn't matter because every element is its own inverse and the group is abelian, so this sum is well defined).  Moreover, the sum is invariant under shifting: shifting at a vertex changes the total by the vertex's degree times the shifting element, which is always zero because the graph is Eulerian. Then it cannot be that $(G,\gamma_G)$ contains $(H,\gamma_H)$ as a full immersion, even though $(G,\gamma_G)$ can contain arbitrary group labeled immersions. It turns out that this is in some sense the only restriction.

We first need some notions from group theory. The \emph{commutator subgroup} of a group $\Gamma$, denoted $[\Gamma, \Gamma]$, is the subgroup generated by all the \emph{commutators}, that is, all elements of the form $a^{-1}b^{-1}ab$ for $a,b \in \Gamma$. 
It is well known that $[\Gamma,\Gamma]$ is a normal subgroup and that $\Gamma/[\Gamma,\Gamma]$ is abelian; in fact, for any normal subgroup $N$ of $\Gamma$, $\Gamma/N$ is abelian if and only if $N$ contains the commutator subgroup. 
We denote by $\Gamma^2$ the set $\{a^2 : a \in \Gamma\}$. We note that $\Gamma^2[\Gamma,\Gamma] = \{ab : a \in \Gamma^2, b \in [\Gamma,\Gamma]\}$ is a  subgroup of $\Gamma$. To see this, modulo $[\Gamma,\Gamma]$ the group is abelian, so the product of two squares is a square and the inverse of a square is a square, implying $\Gamma^2[\Gamma,\Gamma]$ is a subgroup. Since it contains $[\Gamma,\Gamma]$, it is normal as well.  
Finally, we note that if $\Gamma$ is a finite group then $\Gamma/(\Gamma^2[\Gamma,\Gamma])$ is isomorphic to $\mathbb{Z}_2^k$ for some $k \ge 0$; the group is abelian and every element squares to the identity.

\begin{thm}\label{thm:fullImmersion}
Let $\Gamma$ be a finite group. For every Eulerian
%\CM{Is the single vertex graph Eulerian? This is only true if $H$ contains an edge}
$\Gamma$-labeled graph $(H,\gamma_H)$ with at least one edge there exists a $\Gamma$-labeled graph $(H',\gamma_{H'})$ such that for any 4-edge-connected Eulerian $\Gamma$-labeled graph $(G,\gamma_G)$ with no $(H,\gamma_H)$ full immersion, one of the following holds.
\begin{enumerate}
        \item There exists a normal subgroup $\Gamma'$ of $\Gamma$ such that $\Gamma/\Gamma'$ is isomorphic to $\mathbb{Z}_2^k$ for some $k$, and over this quotient labeling the sum of the edge labels in $(G, \gamma_G)$ is different than that of $(H, \gamma_H)$.
        \item $(G,\gamma_G)$ does not contain $(H',\gamma_{H'})$ as an immersion.
\end{enumerate}
\end{thm}

We note as before that in the first option, the sum over the labeling of all edges of $(G,\gamma_G)$ is well defined because the quotient group is isomorphic to $\mathbb{Z}_2^k$.  Since $G$ is Eulerian, this sum is also invariant under shifting.  In fact, in option 1 we may take $\Gamma' =\Gamma^2[\Gamma,\Gamma]$. It is not hard to see that for any Eulerian $\Gamma$-labeled graphs $(G,\gamma_G)$ and $(H,\gamma_H)$ such that option 1 holds, $(G,\gamma_G)$ does not contain $(H,\gamma_H)$ as a full immersion.

\begin{proof}[Proof of Theorem~\ref{thm:fullImmersion}]
    Fix an arbitrary vertex $x \in V(H)$ incident with an edge of $H$ and consider the $\Gamma$-labeled graph $(H^\circ, \gamma^\circ_H)$ obtained from $(H, \gamma_H)$ by adding $2|\Gamma|$ loops of label $a$ at $x$, for each $a \in \Gamma$.
    Applying Theorem~\ref{thm:mainReduction} to $(H^\circ, \gamma_{H^\circ})$ yields a $\Gamma$-labeled graph $(H', \gamma_{H'})$ such that, if $(G, \gamma_G)$ contains $(H', \gamma_{H'})$ as an immersion, then $(G, \gamma_G)$ contains $(H^\circ, \gamma_{H^\circ})$ as an $x$-near-full immersion.
    %\EP{Maybe we can state the theorem like that.}
    If $(G, \gamma_G)$ does not contain $(H', \gamma_{H'})$ as an immersion we immediately conclude.
    We may therefore assume that $(G, \gamma_G)$ contains an $x$-near-full immersion $\phi^\circ$ of $(H^\circ, \gamma_{H^\circ})$.
    For simplicity we also assume that $x = \phi^\circ(x).$  Fix an accommodating labeling $\gamma_G'$ for $\phi^\circ$, and use this labeling for all circuit and trail labels below. Since $G$ is Eulerian, shifting does not change its total label in any elementary abelian $2$-group quotient.  Thus if outcome $1$ holds for the labeling $\gamma_G'$ it holds as well for the original labeling.  
    
If there are unused edges, let $C$ be an Eulerian circuit of their unique non-trivial component, starting and ending at $x$, and fix $\alpha=\gamma_G'(C)$. If there are no unused edges, let $\alpha=1$ and omit $C$ from the concatenations below.
    
Assume to reach a contradiction that option $1$ does not hold.
Our goal is to concatenate $C$, if present, with the branch trails of $\phi^\circ$ corresponding to the added loops, in such a way that the product of all the labels is the identity;
in this case we are done, as we may discard the branch trails of $\phi^\circ$ corresponding to the added loops, obtaining an $x$-near-full immersion $\phi$ of $(H, \gamma_H)$ in the process, append the discarded branch trails to $C$, if present, in a way that gives us a circuit $C'$ starting and ending at $x$ and labeled by the identity, and then append $C'$ to any branch trail of $\phi$ corresponding to an edge of $H$ incident with $x$ in order to obtain $(H, \gamma_H)$ as a full immersion.

Indeed, assuming outcome $1$ does not hold, it must in particular not hold with respect to $\Gamma^2[\Gamma,\Gamma]$.
Over this quotient, the sum of the labels of the edges in the branch trails of $\phi$ equals the total label of $(H,\gamma_H)$, and hence that of $(G,\gamma_G')$. Moreover, the total label of the branch trails corresponding to the added loops is the identity, since each label occurs an even number of times. Consequently, the image of $\alpha$ in $\Gamma/(\Gamma^2[\Gamma,\Gamma])$ is the identity, and therefore $\alpha \in \Gamma^2[\Gamma,\Gamma]$.
    
Let $K=\Gamma^2[\Gamma,\Gamma]$. The subgroup generated by all squares is normal, and its quotient has every element of order at most two, so that quotient is abelian. Thus this subgroup contains $[\Gamma,\Gamma]$, and consequently equals $K$.
Since the squares generate $K$, we may write $\alpha^{-1}=a_1^2\cdots a_t^2$ for some $a_1,\ldots,a_t \in \Gamma$, where $t \leq |K|-1 \leq |\Gamma|-1$ because a shortest such expression has distinct partial products. If $\alpha=1$, we take $t=0$.
For each $i \in [t]$, select two previously unused branch trails of $\phi^\circ$ corresponding to added loops of label $a_i$, and orient both to have label $a_i$ under $\gamma_G'$. There are sufficient copies, since we use at most $2t<2|\Gamma|$ copies of any label. Concatenating these trails in the specified order gives label $\alpha^{-1}$.
For each label, an even number of its added loops remain unused. Pair these loops, and for each pair traverse one corresponding branch trail in its prescribed orientation and the other in reverse, so that the pair contributes the identity.
Concatenate $C$, if present, with the selected trails and then these remaining pairs. The resulting circuit begins and ends at $x$, contains every residual edge and every edge of the added-loop branch trails exactly once, and has label $\alpha\alpha^{-1}=1$, as desired.
\end{proof}

We conclude with the proof of Theorem \ref{thm:fullStructure}, which we restate here.

\fullStructure*

\begin{proof}[Proof of Theorem~\ref{thm:fullStructure}]
    Fix $\Gamma$ and $(H,\gamma_H)$, and let $(H',\gamma_{H'})$
    be the graph supplied by Theorem~\ref{thm:fullImmersion}.
    Let $n=|V(H')|$, and fix a positive integer $k$ such that
    every vertex of $H'$ is incident to at most $k$ edges.
    Let 
    $$t=4kn|\Gamma|^{6+\lfloor\log_2|\Gamma|\rfloor}.$$
    In particular, $t\geq n|\Gamma|$.

    Let $(G,\gamma_G)$ be a $4$-edge-connected Eulerian
    $\Gamma$-labeled graph with no $(H,\gamma_H)$ full immersion.
    If the first outcome of Theorem~\ref{thm:fullImmersion} holds,
    then the first outcome of Theorem~\ref{thm:fullStructure} holds as well.
    Thus we may assume that $(G,\gamma_G)$ does not contain $(H',\gamma_{H'})$
    as an immersion.  Theorem~\ref{thm:MainStructureThm} gives a shifting
    $\tilde{\gamma}$ of $\gamma_G$ and a tree-cut decomposition
    $(T,\mathcal{B})$ of $G$ satisfying the second outcome of
    Theorem~\ref{thm:fullStructure}, completing the proof.
\end{proof}

\section*{Declaration of AI use}

ChatGPT-6 Astra was used in proofreading and editing the final manuscript, identifying a number of typos as well as mathematical issues in the proof of Theorem \ref{thm:fullImmersion} and Lemma \ref{lem:uncrossing}.
%BIBLIOGRAPHY
\bibliography{refs}

\end{document}